\documentclass[12pt,a4paper]{amsart}
\usepackage{amsaddr}
\usepackage{amssymb,amsmath,amsfonts,amsthm,mathrsfs}
\usepackage[dvipsnames]{xcolor}
\usepackage{longtable,booktabs,colortbl,multicol,tikz,enumitem}

\usepackage[
    backend=biber,
    style=alphabetic,
]{biblatex}
\renewbibmacro{in:}{}
\usepackage{graphicx}\usepackage[symbol]{footmisc}

\usepackage{float}

\graphicspath{ {./images/} }

\usepackage{fancyvrb}
\usepackage{comment}
\usepackage{appendix}

\usepackage{geometry}
\usepackage{textcomp}
\usepackage{listings}

\makeatletter

\newcommand\PrologPredicateStyle{}
\newcommand\PrologVarStyle{}
\newcommand\PrologAnonymVarStyle{}
\newcommand\PrologAtomStyle{}
\newcommand\PrologOtherStyle{}
\newcommand\PrologCommentStyle{}

\newif\ifpredicate@prolog@
\newif\ifwithinparens@prolog@

\lst@SaveOutputDef{`_}\underscore@prolog

\newcount\currentchar@prolog

\newcommand\@testChar@prolog%
{%
  \ifnum\lst@mode=\lst@Pmode%
    \detectTypeAndHighlight@prolog%
  \else
    \ifwithinparens@prolog@%
      \detectTypeAndHighlight@prolog%
    \fi
  \fi
  \global\predicate@prolog@false%
}

\newcommand\detectTypeAndHighlight@prolog
{%
  \def\lst@thestyle{\PrologAtomStyle}%
  \ifpredicate@prolog@%
    \def\lst@thestyle{\PrologPredicateStyle}%
  \else
    \expandafter\splitfirstchar@prolog\expandafter{\the\lst@token}%
    \expandafter\ifx\@testChar@prolog\underscore@prolog%
      \ifnum\lst@length=1%
        \let\lst@thestyle\PrologAnonymVarStyle%
      \else
        \let\lst@thestyle\PrologVarStyle%
      \fi
    \else
      \currentchar@prolog=65
      \loop
        \expandafter\ifnum\expandafter`\@testChar@prolog=\currentchar@prolog%
          \let\lst@thestyle\PrologVarStyle%
          \let\iterate\relax
        \fi
        \advance \currentchar@prolog by 1
        \unless\ifnum\currentchar@prolog>90
      \repeat
    \fi
  \fi
}
\newcommand\splitfirstchar@prolog{}
\def\splitfirstchar@prolog#1{\@splitfirstchar@prolog#1\relax}
\newcommand\@splitfirstchar@prolog{}
\def\@splitfirstchar@prolog#1#2\relax{\def\@testChar@prolog{#1}}

\def\beginlstdelim#1#2%
{%
  \def\endlstdelim{\PrologOtherStyle #2\egroup}%
  {\PrologOtherStyle #1}%
  \global\predicate@prolog@false%
  \withinparens@prolog@true%
  \bgroup\aftergroup\endlstdelim%
}

\newcommand\lang@prolog{Prolog-pretty}
\expandafter\lst@NormedDef\expandafter\normlang@prolog%
  \expandafter{\lang@prolog}

\expandafter\expandafter\expandafter\lstdefinelanguage\expandafter%
{\lang@prolog}
{%
  language            = Prolog,
  keywords            = {},      
  showstringspaces    = false,
  alsoother           = @$,
  moredelim           = **[l][\color{teal}]{(}{)},
  MoreSelectCharTable =
    \lst@DefSaveDef{`(}\opparen@prolog{\global\predicate@prolog@true\opparen@prolog},
}

\newcommand\@ddedToOutput@prolog\relax
\lst@AddToHook{Output}{\@ddedToOutput@prolog}

\lst@AddToHook{PreInit}
{%
  \ifx\lst@language\normlang@prolog%
    \let\@ddedToOutput@prolog\@testChar@prolog%
  \fi
}

\lst@AddToHook{DeInit}{\renewcommand\@ddedToOutput@prolog{}}

\makeatother
\definecolor{PrologPredicate}{RGB}{000,031,255}
\definecolor{PrologVar}      {RGB}{024,021,125}
\definecolor{PrologAnonymVar}{RGB}{000,127,000}
\definecolor{PrologAtom}     {RGB}{186,032,032}
\definecolor{PrologComment}  {RGB}{063,128,127}
\definecolor{PrologOther}    {RGB}{000,000,000}

\renewcommand\PrologPredicateStyle{\color{PrologPredicate}}
\renewcommand\PrologVarStyle{\color{PrologVar}}
\renewcommand\PrologAnonymVarStyle{\color{PrologAnonymVar}}
\renewcommand\PrologAtomStyle{\color{PrologAtom}}
\renewcommand\PrologCommentStyle{\itshape\color{PrologComment}}
\renewcommand\PrologOtherStyle{\color{PrologOther}}

\lstdefinestyle{Prolog-pygsty}
{
  language     = Prolog-pretty,
  upquote      = true,
  stringstyle  = \PrologAtomStyle,
  commentstyle = \PrologCommentStyle,
  literate     =
    {:-}{{\PrologOtherStyle :-}}2
    {,}{{\PrologOtherStyle ,}}1
    {.}{{\PrologOtherStyle .}}1
}

\definecolor{OliveGreen}{rgb}{0.8,0.83,0.6}
\definecolor{burntumber}{rgb}{0.54, 0.2, 0.14}
\definecolor{coolblack}{rgb}{0.0, 0.18, 0.39}
\definecolor{darkterracotta}{rgb}{0.8, 0.31, 0.36}
\definecolor{frenchbeige}{rgb}{0.65, 0.48, 0.36}
\usepackage{hyperref}
\hypersetup{colorlinks=true,
linkcolor=NavyBlue,
    filecolor=OliveGreen,      
    urlcolor=burntumber,
    citecolor=darkterracotta,
    anchorcolor=frenchbeige}

    \usepackage{letltxmacro}
    \newcommand*{\SavedLstInline}{}
    \LetLtxMacro\SavedLstInline\lstinline
    \DeclareRobustCommand*{\lstinline}{%
      \ifmmode
        \let\SavedBGroup\bgroup
        \def\bgroup{%
          \let\bgroup\SavedBGroup
          \hbox\bgroup
        }%
      \fi
      \SavedLstInline
    }

\usetikzlibrary{calc}

\newcommand{\Z}{\mathbb{Z}}

\newcommand{\W}{\mathcal{W}}
\newcommand{\E}{\mathfrak{E}}
\newcommand{\BB}{\mathscr{B}}

\newcommand{\lst}{\lstinline}

\usepackage[capitalise]{cleveref}
\newtheorem{thm}{Theorem}[section]
\newtheorem{prop}[thm]{Proposition}
\newtheorem{cor}[thm]{Corollary}
\newtheorem{lemma}[thm]{Lemma}
\crefname{mainthm}{thm}{Main Theorem}
\newtheorem*{mainthm}{Main Theorem}
\theoremstyle{definition}
\newtheorem{defin}[thm]{Definition}

\theoremstyle{remark}
\newtheorem{rmk}[thm]{Remark}
\newtheorem{rmks}[thm]{Remarks}

\calclayout 

\newcommand{\B}{\mathcal{B}}
\newcommand{\C}{\mathcal{C}}

\title[Minimal Lottery designs]{Applying constraint programming to minimal lottery designs}

    \author[Cushing]{David Cushing}
    \address{Department of Mathematics, The University of Manchester, Manchester, UK}
    \email{david.cushing@manchester.ac.uk}
    
    \author[Stewart]{David I. Stewart*}
    \address{Department of Mathematics, The University of Manchester, Manchester, UK}
    \email{david.i.stewart@manchester.ac.uk}
    \let\svthefootnote\thefootnote
\begin{document}

\begin{abstract}
    We develop and deploy a set of constraints for the purpose of calculating minimal sizes of lottery designs. Specifically, we find the minimum number of tickets of size six which are needed to match at least two balls on any draw of size six, whenever there are at most $70$ balls.
    \end{abstract}
    
    \maketitle
    
    \section{Introduction}
    An\let\thefootnote\relax\footnote{* Corresponding author}
    \let\thefootnote\svthefootnote 
    $(n,k,p,t;j)$-lottery design is a hypergraph $H=(X,\B)$ where $X$ denotes a set of vertices of order $n$, and $\B$ a set of subsets of $X$ with size $|\B|=j$; furthermore, $H$ should be $k$-uniform---that is, $|B|=k$ for any $B\in \B$; and for any subset $D$ of $X$ of order $p$, we have $|D\cap B|\geq t$ for some $B\in \B$. More colloquially, $X$ is a set of balls labelled $1$ to $n$, and $\B$ is a set of $j$ tickets each containing $k$ of these numbers, such that for any draw $D$ of $p$ distinct balls from $X$, we can find at least one ticket $B\in \B$ which matches $D$ in at least $t$ places. Mostly because of a consequence of its application to the UK national lottery, this paper is concerned with finding the minimum value $L(n,k,p,t)$ of $j$ when $k=p=6$ and $t=2$. 
    
    In \cite{BateRees}, the numbers $L(n,6,6,2)$ are calculated for $n\leq 54$; however, the proof contains some gaps. We correct and refine the methods of \cite{BateRees} both theoretically and computationally, with the following result:
    
    \begin{mainthm}\label{mainthm}
        The values of $L(n,6,6,2)$ for $n\leq 70$ are listed in  \cref{theconfigs}.
    \end{mainthm}
    The proof of the theorem has two parts. One part is to supply a $(n,6,6,2;j)$-lottery design with $j$ taking the claimed value for $L(n,6,6,2)$ in \cref{theconfigs}; the other, much harder part, is to show that there is no lottery design using fewer blocks (i.e.~tickets). The first part we achieve by describing for each $n$ a  design constructed as a disjoint union of \emph{covering designs}---see \cref{sec:uppbound} for the definition. We discovered that our upper bounds had been found already\footnote{\url{https://lottodesigns.altervista.org/}} but supplied without any rationale; by contrast, the constructions in \cref{theconfigs} are completely transparent. 
    
    Most of this paper is dedicated to the second part of the proof: finding tight lower bounds, the majority of which are new. On rare occasions, we do get lucky by using a general result. The Handbook of Combinatorial Designs \cite[\S32]{Hand} gives the following theorem \cite[Thm.~1]{Furedi} as a highlight, which furnishes us with a highly useful lower bound.

\begin{thm}[Furedi--Sz\'ekely--Zubor]\label{FurediBound} We have 
    \[L(n, k, p, 2) \geq \frac{1}{k}\cdot\min_{\sum_{i=1}^{p-1} a_{i} = n }\left(  \sum_{i=1}^{p-1} a_{i}\left\lceil \frac{a_{i} - 1}{k-1} \right\rceil\right).\]
\end{thm}

Sometimes this lower bound for $L(n,6,6,2)$ is matched by an upper bound coming from covering designs and so $L(n,6,6,2)$ is known without any work. 

    \begin{cor}For $32\leq n\leq 70$, there always exists a $(n,6,6,2;j)$-lottery design with $j$ minimal which is a disjoint union of five $(a_i,6,2)$-covering designs. Hence there exist $\{a_1,\dots,a_5\}$ with $\sum a_i=n$  and
        \[L(n,6,6,2)=\sum_i C(a_i,6,2).\]
    \end{cor}

        
 Perhaps of most interest here is that we make substantial use of the constraint programming library \cite{COC97} in SICStus Prolog\footnote{A free evaluation copy of SICStus can be downloaded from \url{https://sicstus.sics.se/}} \cite{sicstus}. 
 In that respect, this paper complements our earlier work \cite{CSS} as another application of constraint programming to pure mathematics problems; see also \cite{gkkln} for a novel approach using CP to a question about permutation groups. 
 It would be completely out of the question to approach this problem by brute force, which is typically all one has in combinatorics if not deploying CP: to rule out the possibility that $L(70,6,6,2)=34$, one starts with all $\begin{pmatrix}{c}70\\ 6\end{pmatrix}\sim 10^8$ possible tickets, from which one then needs to check every choice of $34$---about $3.4\times 10^{237}$ cases, many orders of magnitude bigger than the number of atoms in the observable universe.
    
\subsection*{The CP Model} Let us give a little more of an overview of the proof of the lower bounds for the main theorem, which will be in practice  the justification of the model for the constraint solver we use; the model is designed to achieve an {\sc\small UNSAT} conclusion when looking for an $(n,6,6,2;j)$-lottery design where $j$ is one less than the value claimed as best possible in \cref{theconfigs}. 

For there is a philosophical principle guiding the design of the model that may be of general interest to constraint programmers, particularly those working with combinatorial designs. By way of giving it some contrast, let us start by imagining a na\"ive model. To rule out $L(70,6,6,2)=34$, one might take a $34\times 6$ matrix of variables, each with domain $\{1,\dots,70\}$---perhaps breaking some symmetry by constraining the matrix to be lex-chained and increasing---and such that each of the $\begin{pmatrix}{c}70\\ 6\end{pmatrix}\simeq 10^8$ possible draws has an intersection of size at least two with at least one of the $34$ tickets. Just to post all these the constraints would need the order of terabytes of data.

Our alternative approach follows the basic thrust of \cite{BateRees}. Let a subset $I\subseteq X$ of the vertices be called an \emph{independent set} if no block of the design intersects it more than once; i.e.~an independent set is a collection of balls no pair of which are found on the same ticket. Then what we do in effect is to localise the constraint problem around a maximal independent set $I$. Our model leverages a tension between the room available in the blocks which intersect $I$ non-trivially---which mostly depends on the number of such blocks, naturally---and the size of $I$. The size of $I$ is important, for if there is an independent set of size $6$ then it exhibits a draw for which our tickets fail: by definition none of our tickets intersects it more than once.

Seen another way, we focus on constraining the topography of the design that surrounds a well-chosen invariant feature of the design---in this case, a maximal independent set. Spinning the design around so as to stare closely at $I$ proves to be, in effect, a powerful symmetry-breaking technique. In fact, it is so good that achieving {\sc\small UNSAT} for the global model requires no labelling of variables that represent the individual vertices of the design; any such labelling work---which is non-trivial, but tractable---is subsumed by the local data in \cref{table:exs} below. 

It seems reasonable to imagine that in another combinatorial setting, a maximal independent set might be profitably substituted by another (hyper)graph-theoretic feature; one would look to localise around this new feature, using CP to do the hard work of enumerating the possible configurations locally, and then show that none can be made to work globally.

    
    \begin{rmk}Let us finally describe the more mundane concern that inspired the paper. In October 2015, the UK national lottery (where $k=p=6$) underwent a change in which the number of balls  $n$ was increased from $49$ to $59$. By way of compensation, a `lucky dip' prize was introduced for matching two balls on a ticket. Consequently, the value of $L(59,6,6,2)$---namely $27$ by our main theorem---is somewhat significant. For the interested reader, \cref{27tix} also lists a minimal set of tickets explicitly.
    
    Having observed that the set of tickets we describe below would have netted \pounds 1810 in the lottery draw of 21 June 2023, the authors were motivated to road-test the tickets in the lottery draw of 1 July 2023; they matched just two balls on three of the tickets, the reward being three lucky dip tries on a subsequent lottery, each of which came to nothing. Since a ticket costs \pounds 2, the experiment represented a loss to the authors of \pounds 54. This unfortunate incident therefore serves both as a verification of our result and of the principle that one should expect to lose money when gambling. 
    
    For a more philosophical discussion of the National Lottery and its implementation for supporting charitable causes, we recommend David Runciman's article in the London Review of Books \cite{runc}.\end{rmk}

\section{Definitions and notation} 
A \emph{hypergraph} $H$ is a pair $(X,\B)$ with $X$ a set and $\B$ a set of subsets of $X$. We will refer to the elements $x\in X$ as \emph{vertices} and the elements $B\in\B$ as \emph{blocks}. The \emph{order} of $H$ is the cardinality of the set $X$. The \emph{size} of $H$ is the number of blocks;~i.e. the cardinality of the set $\B$. A hypergraph is said to be \emph{$k$-uniform} if each $B\in\B$ has order $k$. Note that a $2$-uniform hypergraph is a graph: a block $B=\{x,y\}$ identifies with an edge $x-y$.

\begin{defin}An \emph{$(n,k,p,t;j)$-lottery design} is a $k$-uniform hypergraph $H=(X,\B)$ of order $n$ and size $j$ such that  for any subset $D\subseteq X$ with $|D|=p$, there is at least one $B\in \B$ with $|B\cap D|\geq t$.

An \emph{$(n,k,p,t)$-lottery design} is an $(n,k,p,t;j)$-lottery design in which $j$ is minimal; we denote this minimal integer by $L(n,k,p,t)$.\end{defin}

\begin{rmk}Thinking of $X$ as a set of balls, $\B$ as a set of tickets, and $D$ as a draw, the definition says that in any draw of $p$ balls, at least one ticket matches at least $t$ of the balls in $D$.\end{rmk}

In order to avoid vacuous or incorrect statements we shall always insist $n\geq k\geq t\geq 2$.

We say a block $B$ is an \emph{$x$-block}---or that $B$ is \emph{incident} with a vertex $x$---if $x\in B$. The set $\B_x$ of all $x$-blocks is the \emph{star} of $x$. We define the function \[d:X\to\Z_{\geq 0};\ x\mapsto|\B_x|\] so that $d(x)$ is the \emph{degree} of $x$; i.e.~the number $|\B_x|$ of blocks incident with $x$. More generally, if $I$ is any subset of the vertices, then $B$ is an $I$-block if it is an $x$-block for some $x\in I$. We let $\B_I=\bigcup_{x\in I} \B_x$, and $d(I)=\bigcup_{x\in I} d(x)$ its degree. We will denote by $d_i:=|d^{-1}(i)|$ the number of vertices of degree $i$ in $H$. We also need to analyse the multiset of degrees of vertices in a subset $I\subseteq X$, so we define the function from the power set $\mathcal{P}(X)$ of $X$ to non-decreasing sequences of non-negative integers: \[\delta : \mathcal{P}(X)\to \bigcup_{r=0}^n(\Z_{\geq 0})^r;\ I\mapsto (d(x_1),\dots,d(x_n)),\text{ such that } d(x_1)\leq d(x_2)\leq \dots\leq d(x_n).\]
It will be convenient to assume that the elements of $I$ are always listed in order of their degree, so for example we may have $I=\{x_1,\dots,x_5\}$ such that $\delta(I)=(2,2,3,3,4)$ indicating $d(x_1)=2,\dots,d(x_5)=4$. 

Two vertices $x,y\in X$ are \emph{adjacent} if they are contained in a common block. The set of all vertices adjacent to $x$ is its \emph{neighbourhood} $N(x)$. A subset $I$ of $X$ is an \emph{independent set} or \emph{coclique} if no pair of elements in $I$ are adjacent. An independent set is \emph{maximal} if there is no independent set $J\subseteq X$ with $I\subsetneq J$.

\begin{rmks}\label{indsetrems}(i) If $I$ is an independent set in an $(n, k, p, t; x)$-lottery design, then $|I|<p$ or else there would be a draw $D=I$ matching at most $1$ vertex in any block (noting $t>1$). 
    
(ii) If $I$ is a maximal independent set then we must have $\B_I=X$ or else there is another element $y\in X$ which is not adjacent to any $x\in I$; which would imply there were a larger maximal independent set $I\cup \{y\}$.

(iii) It is easy to see that maximal independent sets need not have the same order. For example, in the graph $x-y-z$, $\{y\}$ is maximal independent, but $\{x,z\}$ is a  maximal independent set of higher cardinality. \end{rmks}

A set $\B'\subseteq\B$ of blocks is \emph{disjoint} if $B_1\cap B_2=\emptyset$ for any $B_1,B_2\in\B'$. In the literature, such a $\B'$ is also referred to as a \emph{matching}. A block $B\in \B$ is \emph{isolated} if $B$ is disjoint from all other blocks in $\B$. (Of course, each vertex of an isolated block has degree $1$.) If $d_0=0$ and each vertex of degree $1$ is contained in an isolated block, then we say $H$ is \emph{segregated}.

We will show later that $H$ may be assumed segregated under the hypotheses of our main theorem. In that case, the interesting analysis reduces to the full subhypergraph induced by the vertices of degree at least $2$. With that in mind, we give a further couple of definitions which will prove central to our argument. Let $I$ be an independent set. Then a \emph{toe}, or more explicitly \emph{$I$-toe} is a vertex $z$ of degree at least $2$ appearing in just one block of $\B_I$. For $x\in I$, we say $z$ is an $x$-toe if it is adjacent to $x$; the set of all $x$-toes is denoted $F_x$, and the union $F_I:=\bigcup_{x\in I} F_x$ is the \emph{foot} or \emph{$I$-foot}. (Note that an $\{x\}$-toe may fail to be an $I$-toe for $\{x\}\subsetneq I$.) Suppose $I=\{x_1,\dots,x_\ell\}$, with $\delta(I)=(d(x_1),\dots,d(x_\ell))$, where $d(x_1)>1$. Then set 
\[\tau(I)=\left(\left|F_I\cap \bigcup\B_{x_1}\right|,\dots,\left|F_I\cap\bigcup\B_{x_\ell}\right|\right),\]
which denotes the distribution of toes among the $x$-blocks as $x$ ranges over the vertices in $I$. 

\section{Preliminaries}
We use this section to collect a miscellany of results on hypergraphs and lottery designs of varying generality that we use in the proof of the main theorem. Our strategy is heavily influenced by the paper \cite{BateRees} and we use most of its results in one form or another. However, we noticed a number of infelicities among the statements and proofs in \cite{BateRees} and so we take the opportunity here to make some corrections. Happily, it follows this paper is largely self-contained. 

\subsection{Upper bounds}\label{sec:uppbound} There are several websites that collect information about the values of $L(n,k,p,t)$\footnote{\url{https://lottodesigns.altervista.org/} is one of the most up to date}; typically one finds there is much better information known about upper bounds than lower bounds, though often little about how these were arrived at. Mostly they can be generated from covering designs.

An \emph{$(n,k,t)$-covering design} is a $k$-uniform hypergraph $H=(X,\B)$ such that every subset of $X$ of size $t$ appears as a subset in at least one block of $\B$; we assume $n\geq k\geq t$. Define $C(n,k,t)$ as the minimal size of an $(n,k,t)$-covering design. Obviously, an $(n,k,t)$-covering design is equivalently an $(n,k,t,t)$-lottery design, so $L(n,k,p,t)\leq C(n,k,t)$. 
\begin{lemma}\label{upperbound} We have 
    \[L(n, k, p, 2) \leq \min_{\sum_{i=1}^{p-1} a_{i} = n}\left(  \sum_{i=1}^{p-1} C( a_{i}, k, 2) \right).\]
\end{lemma}
\begin{proof} Let $H_i=(X_i,\B_i)$ be $(a_i,k,2)$-covering designs of size $C(a_i,k,2)$ such that $X=\bigsqcup_{i=1}^{p-1}X_i$, and let $D$ be a draw of $p$ vertices. Then at least two vertices of $D$ lie in at least one $X_i$, and there is a block $B$ of $\B_i$ containing those two vertices.\end{proof} 

\cref{upperbound} can be deployed using the entries from \cref{coveringnumbers}. The table lists upper bounds for $C(n,6,2)$ for $n\geq 6$ which were harvested from \url{https://www.dmgordon.org/cover/}; these upper bounds are all known to be sharp except when $n=23$ or $24$.
\begin{table}
\begin{tabular}{|c|c||c|c|}\hline
    $n$ & $C(n,6,2)\leq$ & $n$ & $C(n,6,2)\leq$\\\hline
    $6$ & $1$ & $17$ & $12$\\
    $7$ & $3$ & $18$ & $12$\\
    $8$ & $3$ & $19$ & $15$\\
    $9$ & $3$ & $20$ & $16$\\
    $10$ & $4$ & $21$ & $17$\\
    $11$ & $6$ & $22$ & $19$\\
    $12$ & $6$ & $23$ & $21$\\
    $13$ & $7$ & $24$ & $22$\\
    $14$ & $7$ & $25$ & $23$\\
    $15$ & $10$ & $26$ & $24$\\
    $16$ & $10$ & $27$ & $27$\\
    \hline
\end{tabular}\caption{Upper bounds for $C(n,6,2)$}\label{coveringnumbers}
\end{table}

Secondly, it is useful to know that lottery numbers increase mononotonically with $n$.

\begin{lemma}\label{doesntgodown}
    We have $L(n,k,p,t)\leq L(n+1,k,p,t)$.
\end{lemma}
\begin{proof}
    Take an $(n+1,k,p,t;j)$-lottery design $H=(X,\B)$ with $|\B|=j$. Pick $x\in X$ and construct $j$ subsets $\C$ of $X\setminus\{x\}$ from those of $\B$ by replacing $x\in B$ with any other vertex of $X\setminus B$, where necessary. Then it is clear that the hypergraph $H_0=(X\setminus \{x\},\C)$ is an $(n,k,p,t;j)$-lottery design.
\end{proof}

\subsection{Reductions and constraints}

\begin{lemma}\label{BR1}
    Let $H=(X,\B)$ be an $(n, k, p, t; j)$-lottery design  with $j \geq n/k$. Then there exists an $(n, k, p, t; j)$-lottery design $H_0=(X,\B_0)$ with $\bigcup\B_0=X$; i.e.~there are no elements of degree $0$.
        \end{lemma}
        \begin{proof}Let $H$ be a counterexample with $d_0>0$ minimal. Suppose $x\in X$ has degree $0$. Since $jk\geq n$, there must be $y\in X$ with $d(y)\geq 2$. Suppose $y\in B$ for some block $B$ of $\B$ and set $B_0=(B\setminus\{y\})\cup \{x\}$. Now set $H_0=(X,\B_0)$, where $\B_0$ is $\B$ with the block $B$ replaced by $B_0$. Then $H_0$ is a $k$-uniform hypergraph of order $n$ and size $j$. If $D$ is any subset of $X$ of order $p$, then either there is $C\in\B_0$ with $|C\cap D|\geq t$ or we may assume $|B\cap D|=t$, $y\in D$ and $|B_0\cap D|=t-1$. This implies also $x\not \in D$; but then the alternative draw $(D\setminus\{y\})\cup\{x\}$ cannot intersect a block of $\B$ in $t$ elements.\end{proof}
\begin{lemma}\label{segregated}
If $H=(X,\B)$ is an $(n, k, p, t; j)$-lottery design with $n \geq k(p-1)$ and $d_0=0$, then there exists a segregated $(n, k, p, t; j)$-lottery design.
\end{lemma}
\begin{proof}
Suppose there are $r$ isolated blocks in $H$. Then taking one vertex from each yields an independent set $I$, so $r\leq p-1$ by \cref{indsetrems}(i). If $r=p-1$ then the isolated blocks supply all $kr=k(p-1)$ vertices of $X$ and the statement holds. 

Suppose $r=p-2$; then there are $n-k(p-2)$ elements not in isolated blocks. A draw of order $p$ containing one vertex from each isolated block together with two non-isolated vertices can only intersect a non-isolated block in at least $t$ places. Thus the non-isolated blocks must between them contain every pair of the non-isolated vertices. This means that each appears at least twice, or $n-k(p-2)=2$ and they both appear exactly once. But the latter says they are themselves in an isolated block, a contradiction.

Hence we may assume $r\leq p-3$, leaving at least $2k$ elements not in isolated blocks. Let $B$ be a non-isolated block and assume that there are $x,y\in B$ with $d(x)=1$ and $d(y)>1$. We modify $H$ to give a lottery design $H_0$ with $d(y)=1$. By an evident induction, this implies the existence of the required design.

Let $B=B_1,\dots,B_\ell$ be the blocks containing $y$. For $2\leq i\leq \ell$, find a non-isolated element $z_i$ such that $z_i\not\in B_i$; this is possible since there are at least $2k$ non-isolated elements. Form $C_i$ by replacing $y$ with $z_i$ and let $C_1=B_1$. Then we claim we get a new $(n,k,p,t;j)$-lottery design $H_0=(X,\C)$ with $d_0=0$ by letting $\C=(\B\setminus\{B_1,\dots,B_\ell\})\cup\{C_1,\dots,C_\ell\}$.

To prove the claim, take a draw $D$ and assume $D$ does not intersect any block $C\in\C$ in at least $t$ elements. Then we may assume $D$ contains $y$, $|D\cap B_i|=t$ for some $2\leq i \leq l$ and $|D\cap C_i|=t-1$.  Furthermore, since $d(x)=1$, if $x\in D$, then $D\cap B_1=D\cap C_1$ has at least $t$ elements, so we may assume $x\not\in D$. Then replacing $y$ with $x$ in $D$ gives a draw $D_0$ which intersects no block of $\B$ in at least $t$ elements.
\end{proof}

The above results are essentially the same as \cite[Lem.~3.2, Thm.~3.5]{BateRees}. In between is \cite[Lem.~3.4]{BateRees} which shows (correctly) that given an $(n, k, p, 2)$-lottery design with $n>k(p-2)$, then there is another with a maximal independent set of size $p-1$. It is combined with the above two results to claim the existence of a lottery design satisfying the conclusions of all three results. Unfortunately the methods of proof go by altering the degrees of vertices in the design, and it is unclear whether this can be done compatibly. There is a further issue in the proof of \cite[Prop.~4.6]{BateRees} where it incorrectly assumed that an independent set $I$ with $d_2(I)$ maximal can be extended to an independent set of maximal cardinality.

We resolve these issue in \cref{1s2s3s} below, with a stronger result. As in \cite{BateRees}, we make use of Shannon's bound \cite{Sha49} on the chromatic number of a graph, though we get better results by using a sharpness result due to Vizing. (Both Shannon's and Vizing's theorems are given a good exposition in \cite{scheide}.)

Let $G=(V,E)$ be a finite undirected graph possibly with multiple edges between distinct vertices but with no loops. A \emph{$c$-edge colouring} of $G$ is an assignment of a colour to each edge in $G$ such that no two adjacent edges have the same colour and at most $c$ different colours are used. The \emph{chromatic index} $\chi'(G)$ of $G$ is the smallest integer $c$ such that $G$ admits a $c$-edge colouring. From the definition it follows immediately that the maximum degree $\Delta(G)$ is a lower bound for the chromatic index. In fact,
\begin{thm}[Shannon's bound]\label{shannonbound} We have $\chi'(G)\leq \left\lfloor \frac{3}{2}\Delta(G)\right\rfloor$.\end{thm}

For $x,y\in V$, let $\mu_{G}(x,y)$ denote the number of edges between $x$ and $y$. Then a \emph{Shannon graph} of degree $d\geq 2$ is a graph $G$ consisting of three vertices $x,y,z$ such that $\mu_{G}(x,y)=\mu_{G}(y,z)=\left\lfloor\frac{d}{2}\right\rfloor$ and $\mu_{G}(x,z)= \left\lfloor\frac{d+1}{2}\right\rfloor$. For a Shannon graph of degree $d$ we have $\Delta(G)=d$ and $\chi'(H)=\left\lfloor\frac{3}{2}d\right\rfloor$, so 
that the bound in \cref{shannonbound} is sharp. More generally:

\begin{thm}[Vizing]\label{Vizing} Suppose $\chi'(G)=\left\lfloor \frac{3}{2}\Delta(G)\right\rfloor$ where $\Delta=\Delta(G)\geq 4$. Then $G$ contains a Shannon graph of degree $\Delta$ as a subgraph.\end{thm}

For our situation we need a dual version for hypergraphs. Let $H=(X,\B)$ be a hypergraph with blocks of order at most $k$ and let $H_0=(X_0,\B_0)$ of $H$ be a subhypergraph. We say $H_0$ is \emph{$k$-Shannon} if $\B_0=\{B_1,B_2,B_3\}$, $X_0=B_1\cup B_2\cup B_3$, there is at most $1$ vertex in $X_0$ of degree not $2$, and where $|B_1\cap B_2| = |B_2\cap B_3|=\left\lfloor\frac{k}{2}\right\rfloor$ and $|B_1\cap B_3|=\left\lfloor\frac{k+1}{2}\right\rfloor$. If $k=2m$ then this means there are $3m$ vertices in $X_0$ and $H=H_0\sqcup H_1$ is a disjoint union of $H_0$ with another subhypergraph $H_1$. If $k=2m+1$ then there are $3m+1$ vertices of degree $2$ and at most one further vertex $v$, where only $v$ may appear in blocks outside of $\B_0$.

\begin{prop} \label{ninthdeg2}
    Let $H=(X,\B)$ be a hypergraph whose blocks $B\in \B$ are of order at most $k\geq 3$ and suppose $H$ contains precisely $s$ distinct $k$-Shannon subhypergraphs. Then there exists an independent set $I$ containing 
    \[s+\left\lceil\frac{d_2-\left\lfloor\frac{3k+1}{2}\right\rfloor s}{\left\lfloor\frac{3k}{2}\right\rfloor-1} \right\rceil\]
    vertices of degree $2$.
\end{prop}-
\begin{proof}
    Let $H_0=(X_0,\B_0)$ be the subhypergraph of $H$ induced by the vertices not contained in $k$-Shannon subhypergraphs. The sizes of the blocks of $H_0$ are still at most of size $k$. Furthermore $H_0$ contains no $k$-Shannon subhypergraphs. 
    -
    Form a graph $G$ whose vertex set is $\B_0$ and there is an edge between $B_1$ and $B_2$ for each element $x\in B_1\cap B_2$ of degree $2$. Note that $G$ has at most $d_2-\left\lfloor\frac{3k+1}{2}\right\rfloor s$ edges, and contains no Shannon subgraphs of degree $k$. Suppose $\Delta(G)\leq 3$. Then \cref{shannonbound} implies the existence of a $c$-colouring with $c\leq 4\leq \left\lfloor\frac{3k}{2}\right\rfloor-1$. Otherwise we may apply \cref{Vizing} to find a $c$-colouring with $c\leq \left\lfloor \frac{3}{2}k\right\rfloor-1$. 
    
    So in any case, there must exist a set of edges with size at least $$\left\lceil\frac{d_2-\left\lfloor\frac{3k+1}{2}\right\rfloor s}{c}\right\rceil \geq \left\lceil\frac{d_2-\left\lfloor\frac{3k+1}{2}\right\rfloor s}{\left\lfloor\frac{3k}{2}\right\rfloor-1} \right\rceil$$ having the same colour. Any monochromatic set of $m$ edges of $G$ represent $m$ independent vertices of $H_0$ of degree $2$.
    
    Observe that every $k$-Shannon subhypergraph $S$ of $H$ has at least one vertex of degree $2$ which is not adjacent to any vertex outside of $S$; adding these $s$ vertices gives the lower bound in the theorem.
\end{proof}

An elementary rearrangement of the bound above yields the following.

\begin{cor}\label{shanrearrangement}With the hypotheses of the proposition, let $\delta_2$ be the maximum number of independent vertices in $H$ of degree $2$ and let $k=2m+\rho$ with $\rho\in \{0,1\}$. Then 
    \[d_2\leq \delta_2\left(\left\lfloor\frac{3k}{2}\right\rfloor-1\right) + (1 + \rho)s\]
\end{cor}

\begin{prop}\label{1s2s3s} Let $H=(X,\B)$ be a segregated $k$-uniform hypergraph with $|X|=n$, $|\B|=j$, $d_1=kr$, $k=2m+\rho$ and $\rho\in\{0,1\}$. Suppose $H$ contains $s$ distinct $k$-Shannon subhypergraphs. Then there is an independent set $I_0\subseteq X$ containing $r$ vertices of degree $1$ and \[\kappa:=s+\left\lceil\frac{d_2-\left\lfloor\frac{3k+1}{2}\right\rfloor s}{\left\lfloor\frac{3k}{2}\right\rfloor-1} \right\rceil\] vertices of degree $2$. Furthermore, $I_0$ can be extended to an independent set $I\subseteq X$ of order $\pi\geq\kappa+r$ whose vertices have degree at most $3$, provided 
    \[\tag{*} 4n-jk-\left(2 -m\right) s -(3k+m-2)(\pi-1)+(m-2)r>0.\]


\end{prop}
Before giving the proof of the proposition, let us make some elementary observations that are used several times in the sequel. 

Let $H=(X,\B)$ be a $k$-uniform hypergraph with $|X|=n$ and $|\B|=j$. Recall $d_i$ is the number of vertices having degree $i$. Clearly

\begin{equation}\label{setsum}
    \sum_{i\geq 0}d_i=n.
\end{equation}

Let $\BB$ denote the multiset $\bigsqcup B_i$, which contains the vertices of the hypergraph counted with their multiplicities in the blocks $\B$. As there are $j$ blocks we have $|\BB|=jk$. Of course a vertex $x$ will occur in exactly $d(x)$ of the blocks, and so: 

\begin{equation}\label{multisetsum}
    \sum_{i\geq 1}id_i=jk.
\end{equation}

\begin{proof}[Proof of \cref{1s2s3s}] The existence of $I_0$ is immediate from \cref{ninthdeg2} and segregation. Moreover, we may assume $s$ of vertices of $I_0$ each live in distinct $k$-Shannon subhypergraphs, whose union accounts for at most $\lfloor\frac{3k+1}{2}\rfloor\cdot s$ vertices.
    Now take $I\supseteq I_0$ maximal subject to $d(I)\subseteq \{1,2,3\}$ and assume for a contradiction that $|I|\leq\pi-1$. let $\delta_i$ denote the number of vertices of $I$ of degree $i$; note that $\delta_1=r$, again by segregation. 
    
    We now bound from above and below the number $d_2+d_3$ of vertices in $X$ having degrees $2$ or $3$. 
    
     Let $J\subseteq I$ denote set of vertices not in isolated blocks. Then $\B_J$ contains at most $$(2k-1)(\delta_2-s)+\left\lfloor\frac{3k}{2}\right\rfloor s+(3k-2)\delta_3$$ distinct elements. Now if there were a vertex of degree $2$ or $3$ not in $\bigcup\B_J$, then $I$ was not maximal. Therefore,
    \begin{align}d_2+d_3&\leq (2k-1)(\delta_2-s)+\left\lfloor\frac{3k}{2}\right\rfloor s+(3k-2)\delta_3,\nonumber\\
        &\leq (2k-1)(\delta_2-s)+\left\lfloor\frac{3k}{2}\right\rfloor s+(3k-2)(\pi-1-r-\delta_2),\nonumber\\
        \notag&\hspace{7cm} \text{since }\delta_1+\delta_2+\delta_3\leq \pi-1,\nonumber\\
        &\leq (1-k)\delta_2+\left(\left\lfloor\frac{3k}{2}\right\rfloor-2k+1 \right)s +(3k-2)(\pi-1-r),\nonumber\\
        &\leq (1-k)\delta_2+\left\lfloor\frac{2-k}{2}\right\rfloor  s +(3k-2)(\pi-1-r)
               \label{combine}
    \end{align}

    On the other hand, consider the multiset $\BB=\bigsqcup_{B\in\B} B$, of order $jk$. The $I$-blocks contribute exactly $\delta_1k+2\delta_2k+3\delta_3k$ of these elements. Since there are $2d_2+3d_3$ elements of degrees $2$ or $3$, there are $2\delta_2k+3\delta_3k-2d_2-3d_3$ elements of $\BB$ of order at least $4$ coming from the $I$-blocks. The remaining $(j-\delta_1-2\delta_2-3\delta_3)$ blocks all consist of elements of order at least $4$. Thus the multiset $|\BB_{\geq 4}|$ of all elements of degrees $4$ and above has order
    \[|\BB_{\geq 4}|=(j-\delta_1-2\delta_2-3\delta_3)k + 2\delta_2k+3\delta_3k-2d_2-3d_3 =jk-rk-2d_2-3d_3.\] 
    Since $|\BB_{\geq 4}|=\sum_{i\geq 4} id_i$, we get 
    \[\sum_{i\geq 4} d_i \leq \frac{jk-rk-2d_2-3d_3}{4}.\]
    So \[n=\sum d_i\leq rk+d_2+d_3+\frac{jk-rk-2d_2-3d_3}{4}.\] 
    
    Rearranging gives $d_2+d_3\geq 4n-jk-d_2-3rk$,   
    which combines with (\ref{combine}) to give 
    \begin{eqnarray*} (1-k)\delta_2+\left\lfloor\frac{2-k}{2}\right\rfloor s +(3k-2)(\pi-1-r) \geq 4n-jk-d_2-3rk;\\
        \text{i.e.}\quad  d_2+(1-k)\delta_2 \geq 4n-jk-\left\lfloor\frac{2-k}{2}\right\rfloor s -(3k-2)(\pi-1)-2r.\end{eqnarray*}
    
    Now using \cref{shanrearrangement} we get
    \begin{align*}
        \delta_2\left(\left\lfloor\frac{3k}{2}\right\rfloor-1\right)& + (1 + \rho)s +(1-k)\delta_2\geq 4n-jk-\left\lfloor\frac{2-k}{2}\right\rfloor s -(3k-2)(\pi-1)-2r\\
        \left\lfloor\frac{k}{2}\right\rfloor\delta_2  &\geq 4n-jk-\left(\left\lfloor\frac{2-\rho}{2}\right\rfloor+ 1 + \rho-m\right) s -(3k-2)(\pi-1)-2r\\
        \left\lfloor\frac{k}{2}\right\rfloor\delta_2  &\geq 4n-jk-\left(2 -m\right) s -(3k-2)(\pi-1)-2r.
    \end{align*}
    Using $\delta_1+\delta_2\leq \pi-1$, with equality if and only if $\delta_3=0$, then
    \[\left\lfloor\frac{k}{2}\right\rfloor(\pi-1-r)  \geq 4n-jk-\left(2 -m\right) s -(3k-2)(\pi-1)-2r;\]
    so
    \[4n-jk-\left(2 -m\right) s -(3k+m-2)(\pi-1)+(m-2)r\leq 0,\]
    which is a contradiction, proving the 
    proposition.
    \end{proof}

The following results are all used to give constraints to Prolog. The first two respectively bound above and below the number of isolated blocks.

\begin{lemma}\label{rbound}
    Let $(X, \mathcal{B} )$ be a segregated $(n, k, p, 2; j)$-lottery design. Suppose that $\B$ has $r$ isolated blocks. Then,
    $$ r \geq \left\lceil\frac{2n-jk}{k} \right\rceil. $$
 \end{lemma}
\begin{proof}We have $jk=d_1+\sum_{i\geq 2} id_i \geq d_1+2(n-d_1)$. Write $d_1=rk$ and rearrange to get the formula above.
\end{proof}

\begin{lemma}\label{rboundforfuredi}Let $H=(X,\B)$ be a segregated $(n, k, p, t; j)$-lottery design with $k=2m$, at least $r$ isolated blocks and at least $s$ disjoint $k$-Shannon subhypergraphs. Then 
\[L(n-2mr-3ms,k,p-s-r,t)\leq j-r-3s\]    
\end{lemma}
\begin{proof}Suppose $B_1,\dots,B_r$ are isolated and $H_1\dots H_s$ are the disjoint $k$-Shannon subhypergraphs of $H$, with $H_i=(X_i,\{C_{i1},C_{i2},C_{i3}\})$. Let $Y$ be the remaining vertices inducing a subhypergraph $H_0=(Y,\B_0)$ of $H$ where $|\B_0|=j-r-3s$. Fix one vertex $v_i\in B_i$ for $1\leq i\leq r$ and $v_{r+1}\dots v_{r+s}$ in each of the $C_{i1}$ with $1\leq i\leq s$. For any choice $\{v_{r+s+1},\dots,v_p\}$ of vertices from $Y$, the draw $D=\{v_1,\dots,v_p\}$ intersects with some $B\in \B$ in at least $t$ vertices. By construction of $D$, $B$ is neither isolated nor one of the $C_{ij}$. But this means $B$ must match $t$ of the remaining $p-r-s$ elements of $D$. Hence $H_0$ is an $(n-2mr-3ms,k,p-s-r,t;j-r-3s)$-lottery design, which implies the inequality as shown.\end{proof}

\begin{lemma}\label{lowerboundd2}Let $H$ be a segregated $k$-uniform hypergraph. Then 
    \[d_2\geq 3n-2rk-jk\]
\end{lemma}
\begin{proof}We recall $\sum d_i=n$ and $\sum id_i=jk$, which implies 
    \[jk-n=d_2+\sum_{i\geq 3}(i-1)d_i\geq d_2+2\sum d_i=-2d_1-d_2+2n.\qedhere\]\end{proof}

\begin{lemma}\label{minnumiblocks}
    Let $(X,\B)$ be a $k$-uniform hypergraph with maximal independent set $I$. Then 
    $$|\B_I|\geq \left\lceil\frac{n-p+1}{k-1}\right\rceil.$$
\end{lemma}
\begin{proof}
    We must have $\bigcup \B_I=X$ or $I$ is not maximal. For $x\in I$, two blocks $B,C\in\B_x$ intersect in at least $x$ so $|\bigcup \B_x|\leq (k-1)|\B_x|+1$. Since $|I|\leq p-1$, then summing over $x\in I$ yields $n\leq (k-1)|\B_I|+p-1$.
\end{proof}

\subsection{Excess, toes and webbings}\label{sec:toes}For the values of $(n,k,p)$ under consideration, one tends to find $L(n,k,p,2)\sim n/2$. Thus, on average, the degree of a vertex in an $(n,k,p,2)$-lottery design is about $2$. The next definition follows \cite{BateRees} with a view to bounding the extent of departure from this average value.
\begin{defin}For a set of vertices $Y\subseteq X$,  the \emph{excess} of $Y$ is the sum \[\E(Y)=\sum_{i>2}(i-2)\cdot |d^{-1}(i)\cap Y |.\]\end{defin}

Note that if $Y=X$, we get $\E(X)=\sum_{i>2}(i-2)\cdot d_i$. The following gives an easy characterisation of $\E(X)$.

\begin{lemma}\label{excessfromisolated} Let $H=(X, \mathcal{B} )$ be a segregated hypergraph with $r$ isolated blocks. Then
    \[\E(X)=jk+rk-2n.\]\end{lemma}
   \begin{proof}We have $jk=\sum id_i$. Since $n=\sum d_i$, we get \[jk-2n=-d_1+\sum_{i\geq 3}(i-2) d_i=-d_1+\E(X).\qedhere\]
\end{proof}

The possible number of toes is constrained by the value of the excess $\E(X)$, in a manner we now describe. First, the following is \cite[Lem.~4.7]{BateRees} after the removal of a significant typo. 

\begin{lemma}\label{mintoes}
    Let $(X, \mathcal{B} )$ be a segregated $(n, k, p, 2; j)$-lottery design with maximal independent set $I$ and $F_I$ its foot. Suppose further that $\B$ has $r$ isolated blocks. Then,
    $$ |F_I| \geq 2n - 2p +2 -(k-1)(|\B_I|+r).$$
\end{lemma}
\begin{proof}As $I$ is maximal, we have $\bigcup\B_I=X$. The multiset \[\mathscr{R}:=\bigsqcup_{x\in I,d(x)>1,B\in \B_x} B\setminus\{x\}\] contains the $kj-rk-|I|$ vertices which are neither isolated nor contained in $I$, and by definition, the toes are the elements of multiplicity $1$ in $\mathscr{R}$. Therefore each of the remaining $kj-rk-|I|-|F_I|$ vertices appears at least twice in $\mathscr{R}$. Since $|\mathscr{R}|=(k-1)(|\B_I|-r)$, we have 
    \[|F_I|+2(n-rk-(|I|-r)-|F_I|)\leq (k-1)(|\B_I|-r).\]
Rearranging and using $|I|\leq p-1$, we get the inequality as claimed.
\end{proof}

Now suppose $x\in I$ for $I$ a maximal independent set, with $F_x$ the set of $x$-toes. Let us assume 
\begin{equation*}|F_x|\geq k.\tag{*}\end{equation*}
Then it follows there are at least two blocks $B,C\in \B_x$, say, containing (necessarily distinct) toes; say $y\in B$ and $z\in C$. If $y$ and $z$ were not adjacent, then replacing $x$ with $y,z$ in $I$ would yield a larger independent set, a contradiction. Thus there is a block $W$ with $y,z\in W$. We refer to such blocks as \emph{webbings}. Formally, $W$ is an \emph{$x$-webbing} if $W\not\in\B_I$ and $W$ contains distinct $x$-toes; the set of $x$-webbings is later denoted $\W_x$. Note that under the assumption (*), each toe appears at least once in a webbing, so that it must have degree at least $2$. 

More precisely, suppose there are $\tau_i$ toes in distinct $x$-blocks $B_i$ with $\tau_1\geq\tau_2\dots\geq\tau_s\geq 1$. Then each of the $\sum_{i<j}\tau_i\tau_j$ pairs of toes must appear in some webbing. In the case $k=6$ and $|F_x|\geq 7$, one can see that some toes must have degree at least $3$, for example. This implies non-trivial lower bounds on the excess $\E(F_x)\leq \E(X)$.

\begin{lemma}\label{toetable}
    Let $x\in X$ with $x$ of degree $2$ or $3$ and $F_x$ the set of $x$-toes. The table below gives minimum values of $\E(F_x)$ in terms of $|F_x|$.
    \begin{center}
    \begin{table}[ht]
    \begin{tabular}{ c|c|c|c|c|c|c|c|c|c|c|c| } 
    $|F_x|$ & $\leq 5$ & $6$ & $7$ & $8$ & $9$ & $10$ & $11$ & $12$ & $13$ & $14$ & $15$\\\hline
    $\min(\E(F_x))$ & $0$ & $0$ & $2$ & $3$ & $7$ & $10$ & $11$ & $12$ & $20$ & $25$ & $27$
    \end{tabular}\bigskip
    \caption{Minimal excesses implied by numbers of toes\label{table:exs}}
    \end{table}
    \end{center}
If moreover $F_x$ is known to contain no elements of degree $2$, then $\min(\E(F_x))\geq |F_x|$.\end{lemma}
\begin{proof}For the second statement of the lemma, just observe that any $x$-toe of degree at least $3$ contributes at least $1$ to $\E(F_x)$.

For the table itself, suppose $|F_x|=\tau_1+\tau_2+\tau_3$ is a partition of $|F_x|$ into summands of size at most $6$. If $|F_x|\leq 6$, then one webbing $W$ suffices to cover all toes, and so the minimum possible excess of $0$ is a achieved by a configuration of $x$-blocks and one webbings in which each toe appears just twice. Otherwise suppose there are $w>1$ webbings, containing each of the $\tau_1\tau_2+\tau_1\tau_3+\tau_2\tau_3$ pairs of toes, so that 
\[\left\lceil\frac{\tau_1\tau_2+\tau_1\tau_3+\tau_2\tau_3}{3}\right\rceil.\] is an upper limit for $w$. 

To save some time, we used the powerful linear programming solver Gurobi to solve the following problem. let $M$ be a $w\times |F_x|$-matrix of variables taking values in $0$ and $1$, with the rows representing webbings and a $1$ appearing in the $(i,j)$-th entry if a toe labelled $j$ is in the $i$th webbing. Since $|F_x|>6$, we know that toe $j$ appears once in the $x$-blocks and at least once in the webbings, so $\E(F_x)+|F_x|$ is the sum $S$ of all entries of the $M$. Furthermore, the rows of $M$ must all sum up to integers less than or equal to $k=6$, and columns $j_1$ and $j_2$ must have scalar product at least $1$ whenever $j_1$ and $j_2$ come from different parts of the partition \[\{1,\dots,n\}=\{1,\dots,\tau_1\}\cup\{\tau_1+1,\dots,\tau_1+\tau_2\}\cup\{\tau_1+\tau_2+1,\dots,\tau_1+\tau_2+\tau_3\}.\] We ask Gurobi to minimise the sum $S$ subject to these constraints, and it results in the table in the lemma. 

Since Gurobi only gives answers up to a percentage accuracy, its output does not amount to a proof of optimality, and so we wrote some additional Prolog code to check that the values of $\min(F_x)$ one below those in the table are infeasible.\footnote{Code is available at \url{github.com/cushydom88/lottery-problem}} 
\end{proof}

\begin{rmks}(i) This improves the bound in \cite[Table 1]{BateRees} for $\min(|F_x|)$ when $|F_x|=13$ from $16$ to $20$. Since Gurobi outputs a feasible configuration of webbings for each value of $|F_x|$, we know that the minima can be achieved.

(ii) Establishing the values in \cref{table:exs} is by far the most computationally intensive task required in the paper---for the value $n=15$, the solver took about $6$ hours to rule out the possibility of an excess of $26$. By contrast, the main Prolog program that uses this precomputed data only takes about $2$ minutes to do all the cases from $n=30$ to $n=70$ on an M1 Macbook Air, running Sicstus Prolog (currently only availabel on an Intel build) in emulation mode. It is highly likely another CP solver running natively on a better computer would do this in a fraction of the time.\end{rmks}

The following rather specific two results lead to some surprisingly effective constraints. The first is an easy check left to the reader.

\begin{lemma}\label{reduce3to2}
    Let $H=(X,\B)$ be a hypergraph and $I\subseteq X$ an independent set. Let $x,y\in I$ with $d(x)=2$, $d(y)=3$ and $z$ another vertex of degree $2$ adjacent to $x$ and $y$. Suppose there exists a toe $w$ of degree $2$ adjacent to $x$ but not adjacent to $z$. Then $I'=I\setminus \{x,y\}\cup\{w,z\}$ is an independent set. 
\end{lemma}

\begin{lemma}\label{adjacency}
    Let $H=(X,\B)$ be a hypergraph and $I\subseteq X$ an independent set of maximal order. Let $x,y\in I$ with $d(x)=2$, $d(y)=3$ and $z$ another vertex of degree $2$ adjacent to $x$ and $y$. Suppose there exists a toe $w$ adjacent to $x$ but not adjacent to $z$. Let $v$ be a toe adjacent to $y$ but not adjacent to $z$. Then $w$ and $v$ are adjacent.
\end{lemma}
\begin{proof} Suppose $w$ and $v$ are not adjacent. Then $I'=I\setminus\{x,y\}\cup\{v,w,z\}$ is independent of larger order than $I$. 
\end{proof}

In the following let $\tau_r=|\tau^{-1}(r)|$ be the number of vertices $x$ in $I$ such that $|F_x|=r$. Say an independent set $I$ is \emph{$2$-max} if it is of maximal order and contains a maximal order subset of independent vertices of degree $2$.


\begin{lemma}\label{d3excess}
    Let $H=(X,\B)$ be a $(n,6,6,2;j)$-lottery design containing $r$ isolated blocks and $s$ Shannon subhypergraphs. Let $I$ be a $2$-max independent set whose vertices have degree at most $3$, containing $\delta_2$ vertices of degree $2$. Then 
    \[6j-3n+s+8\delta_2+12r\geq 7\tau_{13} +11\tau_{14}+12\tau_{15}.\]
\end{lemma}
\begin{proof} For $Y\subseteq X$, let $\mathfrak{F}(Y)=\sum_{i\geq 4}(i-3)|d^{-1}(i)\cap Y|$. In particular,
    \[\mathfrak{F}(X)=\sum_{i\geq 4}(i-3)d_i=\E(X)-n+rk+d_2=6j+12r-3n+d_2,\] 
    where we use \cref{excessfromisolated}.  Let $x\in I$ with $\tau(x)=13$. Then all $x$-toes have degree at least $3$: one observes that each toe $t\in F_x$ is opposite at least $8\geq 6$ others; thus there must be at least two $x$-webbings containing. From \cref{toetable} we have $\E(F_x)\geq 20$. Thus \[\mathfrak{F}(F_x)\geq \E(F_x)-13\geq 20-13=7.\] Similar arguments for $\tau(x)=14$ or $15$ yield \[\mathfrak F(X)\geq 7\tau_{13} +11\tau_{14}+12\tau_{15}.\]
    Since $I$ is $2$-max, \cref{shanrearrangement} gives $d_2\leq 8\delta_2+s$ and we are done.
\end{proof}

\begin{lemma}\label{changingsocks}
    Let $H=(X,\B)$ be a $(n,6,p,2;j)$-lottery design containing $r$ isolated blocks and $s$ Shannon subhypergraphs. Let $I$ be an independent set of order $p-1$ and take $x\in I$ with $d(x)=3$, $\tau(x)=12$ and $\E(F_x)\leq 14$. Then at least one of the following holds:
    \begin{enumerate}
        \item there exists $y\in X$ with $d(y)=3$ such that $I'=I\setminus\{x\}\cup\{y\}$ is independent with $\tau(y)\leq 11$;
        \item there exists an $(n-14,6,p-2,2;j-7)$-lottery design.
    \end{enumerate}
\end{lemma}
\begin{proof}
    Assume (1) does not hold.

    It can be shown that the hypotheses imply, up to isomorphism, that 
    \[\B_{x}=\{\{x,1,2,3,4,\_\},\{x,5,6,7,8,\_\},\{x,9,10,11,12,\_\}\},\text{ and}\]
    \[\W_x\supseteq \{\{1,2,5,6,9,10\},\{1,2,7,8,11,12\},\{3,4,5,6,11,12\},\{3,4,7,8,9,10\}\},\]
where $F_x=\{1,\dots,12\}$. Since $\E(F_x)\leq 14$, at least one of the toes in $F_x$ must have degree $3$; say, with label $1$. 
Now
$\B_1=\{\{1,x,2,3,4,z\},\{1,2,5,6,9,10\},\{1,2,7,8,11,12\}\}.$ Since (1) does not hold, it follows that $z$ is a toe. Moreover,
\[\mathcal{W}_1\supseteq \mathcal{V}_1:=\{\{x,z,5,6,7,8\},\{x,z,9,10,11,12\},\{3,4,5,6,11,12\},\{3,4,7,8,9,10\}\}.\]
Hence the vertices $Y:=\{x,z,1,2\dots,12\}$ induce a subhypergraph of $H$ containing a $(14,6,2,2;7)$-lottery design with blocks $\C=\B_1\cup\mathcal{V}_1$. For each block $B\in\B\setminus\C$, replace any occurrence of a vertex in $Y$ with a vertex in $X\setminus (Y\cup B)$. For any draw of size $p-1$ from $X\setminus Y$, create one of size $p$ by appending the vertex $1$ to it. This must intersect a block of $\B\setminus\C$ in at least $2$ vertices and hence gives the lottery design as specified in (2).
\end{proof}

\section{Proof of the main theorem}
We now indicate how the constraints established in the previous section are transcribed into Prolog code for the constraint solver to establish the theorem on our behalf. In summary, the CP solver works upwards through values of $n$, first finding the best upper bound on the size of $\B=j$ arising from a disjoint union of five covering designs; often the solver discovers that $j$ is the same as the previous number $L(n-1,6,6,2)$ and from \cref{doesntgodown}, we know this must be optimal. Otherwise the CP solver assumes the existence of a lottery design with $|\B|=j-1$ and looks to bind variables to values describing viable configurations for the blocks connected to a maximal independent set $I$; in order words.~the $I$ blocks. By configuration, we mean: 
\begin{enumerate}\item the degrees of the vertices in the independent set---in other words, the number of $I$-blocks incident with each vertex; these values must obey the constraints established above, for example \cref{ninthdeg2}.
    \item the distribution of toes among the $I$-blocks, which must obey the constraints established in \cref{sec:toes}.
\end{enumerate}
In each case, the solver returns {\small\sc unsat} and the theorem is proved. After we give a detailed account of the Prolog code in the next section, we give an illustrative example in case $n=47$,

\subsection{SICStus Prolog code}
The code is available to download from \url{github.com/cushydom88/lottery-problem}, which implements the constraints described in the previous section. We give a description of the strategy and its functionality. 

The user first loads SICStus Prolog and consults the file \lstinline{lottery.pl} through the command 
\begin{lstlisting}
    ?- ['$PATH_TO_DIRECTORY/lottery.pl'].
\end{lstlisting} 
Then one queries Prolog at the command line by asking it to solve (for any unbound variables) in a conjunction of predicates. The following is an example of the output produced from the main predicate in our code:
\begin{figure}[ht]\begin{lstlisting}
    | ?- lottery_numbers_in_range( 69, 71).
%  L(69,6,6,2) = 35
%  L(70,6,6,2) = 35
% Conjecture L(71,6,6,2) = 38. 
% Bad RSTuples [R,S,A,B]: 
% [0,0,0,32]
% [0,1,9,33]
% [0,2,18,34]
% [1,0,0,24]
% Delta(I) exceptions [R,S,D2L,D2U,Delta]:
% [0,0,0,0,[3,3,3,3,3]]
\end{lstlisting}\caption{Example output}\label{outputexample}\end{figure} 

The principal predicates are:

\lstinline{lottery_numbers_in_range( NMin, NMax)}. This predicate writes output to the terminal of the form given in \cref{outputexample}; i.e.~it works sequentially with $n$ from the value of \lstinline{NMin} to the value of \lstinline{NMax} either outputting a line stating the value of $L(n,6,6,2)$ or a conjectured value of it, which is correct modulo a list of exceptional cases which could perhaps be checked by hand. The conjectured value for \lstinline{NMin} is first computed from scratch, using a lower bound of $1$; for efficiency, from that point onwards, it then passes the (in some cases, conjectured) value of $L(n,6,6,2)$ as a lower bound for $L(n+1,6,6,2)$, invoking \cref{doesntgodown}. 

\lstinline{upper_bound( N, Guess, UB )} is called by the previous predicate, and is checked recursively. It holds when \lstinline{N}~$=n$ and the integer \lstinline{Guess} can be achieved as a sum of at most $p-1=5$ values $C(a_i,6,2)$ in \cref{coveringnumbers}. Then \lstinline{UB} is bound to \lstinline{Guess}. Otherwise it is declared that \lstinline{upper_bound( N, Guess+1, UB )} should hold.

This predicate is first asserted with \lstinline{Guess} as the (conjectured) value $j_{n-1}$ of $L(n-1,6,6,2)$. In many cases, it turns out that $j_{n-1}$ can be achieved as a sum of at most $5$ values $C(a_i,6,2)$ and so \lstinline{UB} is bound to $j_{n-1}$. In that case, by \cref{doesntgodown}, we conclude immediately $L(n,6,6,2)=L(n-1,6,6,2)$---modulo any previous cases that remain to be checked by hand.

Otherwise, \lstinline{UB} $>j_{n-1}$ and we seek to rule out $L(n,6,6,2)=$ \lstinline{UB}$-1$. Assume therefore, in search of a contradiction, that there is a lottery design $H=(X,\B)$ with $|\B|=$~\lstinline{UB}$-1$. Then further predicates are engaged which either generate the sought contradiction, or progressively collect a list of information about possible cases that cannot be ruled out.

\lst{bound_isolated_blocks_and_num_shans( N, UB, Rs, RSPairs )} takes the value of \lstinline{UB} from the above predicate and binds \lstinline{RSPairs} to a list of plausible pairs $(r,s)$ where $r=d_1/6$ is the number of isolated blocks in $H$, and $s$ is the number of Shannon subhypergraphs in $H$---a pair will be determined as plausible by the following: Assume $H$ is a $(n,6,6,2;\lstinline{UB}-1)$-lottery design with $(r,s)$ its number of isolated blocks and Shannon subhypergraphs. Then $r+s\leq 5$ (or else there exists $I$ of order $6$); $r$ must satisfy the inequality in \cref{rbound}; and $(r,s)$ must satisfy the inequality in \cref{rboundforfuredi}, implying there exists a $(n-6r-6s,6,6-s-r,2;\lstinline{UB}-1-r-3s)$-lottery design which in turn constrains $r$ and $s$ by appeal to \cref{FurediBound}.

\lstinline{get_deltaI_exceptions( N, UB, MinNumIBlocks, [R,S], DeltaExceptions ) }. Here we assume that the bound labelled (*) in \cref{1s2s3s} holds with $\pi=5$ and consider independent sets $I$ satisfying its conclusion. We find all possibilities for $\delta(I)$ satisfying a long list of constraints. More specifically, for any given pair \lstinline{[R,S]} in \lstinline{RSPairs}, this predicate binds \lstinline{DeltaExceptions} to a list whose entries are tuples \[\lstinline{DeltaException=[R,S,D2L,D2U,Delta]}\] where \lstinline{Delta} represents a tuple $\delta(I)$ satisfying our system of  constraints. It is accompanied by the extra data \lstinline{D2L} and \lstinline{D2U}, which are lower and upper bounds for the value of $d_2$---these arise as by-products of the calculations we describe below. \lstinline{Delta} always takes the following form:
\begin{enumerate}\item the first $r=\lstinline{R}$ elements are $1$;
    \item the next $s=\lstinline{S}$ elements are $2$;
\item the next $q$ elements are each $2$, where $q$ is the minimum value implied by \cref{ninthdeg2};
\item the remaining $5-r-s-q$ elements of \lstinline{Delta} are either $2$ or $3$;
\item the sum of \lstinline{Delta}, equal to $|\B_I|$, satisfies the bound in \cref{minnumiblocks} with $\pi=5$;
\item letting $J$ be the last $5-r-s$ elements of $I$, and \lstinline{DeltaTail} the corresponding sublist $\delta(J)\subseteq\delta(I)$, then \lstinline{DeltaTail} satisfies the predicate \lstinline{can_populate_toes_in_Iblocks} as we describe next.
\end{enumerate}

\lstinline{can_populate_toes_in_Iblocks(N, UB, R, S, DeltaTail, [DeltaTail, D2L, D2U])}. This  binds a variable \lst{Excess} to the value $\E(X)=6(\lstinline{UB}-1+\lstinline{R})-2\lstinline{N}$ (see \cref{excessfromisolated}), and a variable \lst{MinToes} to the lower bound on $|F_{J}|$ supplied by \cref{mintoes}. Suppose there are $\delta_3$ vertices in $I$ of degree $3$. Each contributes $1$ to $\E(X)$, and since these vertices are not toes by definition, the contribution to the excess $\E(X)$ from toes must be at most $\E(X)-\delta_3$. Hence a variable \lst{FootExcess} is created and bound to \lst{Excess-NumThrees}. 

\lstinline{populate_toes_in_Iblocks( DeltaTail, MinToes, Excess, Vs )} binds \lstinline{Vs} to a solution of a  constraint problem to determine $\tau(J)$; in other words, to determine viable distributions of toes among the $J$-blocks. (If it fails, the case (\lstinline{Delta,R,S}) will not feature in the list \lstinline{DeltaExceptions}.) The constraint problem is as follows: \begin{enumerate}
\item \lstinline{Vs} is a sequence of variables taking values $(\tau(x_{1}),\dots,\tau(x_{5-r-s}))$ where $J=\{x_1,\dots,x_{5-r-s}\}$ and the $i$th entry in \lstinline{DeltaTail} is $d(x_i)$; 
\item the total number of toes $\sum\tau(J)$ represented by \lstinline{sum(Vs)} is at least the calculated lower bound \lstinline{MinToes};
\item the sum over $1 \leq i\leq 5-r-s$ of the minimum values for  $\E(F_{x_i})$ implied by \cref{toetable} is at most \lst{FootExcess}.
\end{enumerate}

The solutions \lstinline{Vs} are then compelled to satisfy \lstinline{d3_excess_check}, which implements the constraint \cref{d3excess}; and \lstinline{twos_lie_with_twos}, which implements the constraints from \cref{reduce3to2} and \cref{adjacency}. Finally \lstinline{changing_socks} implements \cref{changingsocks} repeatedly.

\lstinline{get_bad_RS_tuples(N, UB, RSPairs, BadRSTuples1 )}. Here we suppose the inequality (*) in \cref{1s2s3s} does not hold with $\pi=5$. 

This predicate binds \lstinline{BadRSTuples1} to a list of tuples \lstinline{[R,S,A,B,C]} where:

\begin{enumerate}\item \lstinline{[R,S]} is a member of \lstinline{RSPairs} representing the pair $(r,s)$ such that the bound (*) in \cref{1s2s3s} fails, with $j=\lstinline{UB}-1$;


\item \lstinline{A} is a lower bound for $d_2$ implied by  \cref{lowerboundd2}; and 

\item \lstinline{B} is an upper bound for $d_2$ achieved by setting $\delta_2=4-r$ in \cref{shanrearrangement}.\end{enumerate}

Take a tuple \lstinline{[R,S,A,B]} in  \lstinline{BadRSTuples1}, and assume it satisfies the bound in \cref{1s2s3s} with $\pi=4$. Then it is possible to find an independent set $I$ of size $4$ with \lstinline{R} elements of degree $1$ and at least \lstinline{S} of degree $2$ and the rest of degree $2$ or $3$. For every value of $d_2$ lying between $B$ and $C$ a list of possible values of $d_3$ are calculated. For each such pair $(d_2,d_3)$, we then solve a constraint problem to find solutions for $\delta(I)$---or $\delta_2$, which amounts to the same thing. More specifically, we let \lstinline{DeltaTail} represent the last $4-r-s$ entries of $\delta(I)$ with $\delta_2-s$ of degree $2$ and $\delta_3$ of degree $3$, corresponding to the subset $J\subseteq I$.  We have $|\BB_{J}\setminus J|=2(\delta_2-s)+3\delta_3$. These must contain at least one each of the $d_2-9s$ vertices of degree $2$ not in Shannon subhypergraphs and each of the $d_3$ vertices of degree $3$. Hence we may bound below the size of the $I$-foot by
\begin{align*}|F_{I}| \geq 2(d_2 -9s &+ d_3 - 4 + r + s) - 10(\delta_2-s) - 15\delta_3\\ &= 2d_2 - 6s + 2d_3 - 8 + 2r - 10\delta_2 - 15\delta_3.\end{align*}
Now $\E(F_{I})\geq d_3-\delta_3$. Finally the predicate \lstinline{populate_toes_in_Iblocks} is invoked. If there are no solutions for any pair $(d_2,d_3)$, then the tuple is excluded from  \lstinline{BadRSTuples1} to form  \lstinline{BadRSTuples}.

The list \lstinline{BadRSTuples} is displayed on the user output stream. (In case $n\leq 70$ this list is empty.)

\subsection{Example: \texorpdfstring{$n=47$}{n=47}}
Suppose $n=47$ and the previous value $L(46,6,6,2)=16$. Since a disjoint covering design configuration exists with $17$ blocks, we must rule out the possibility of a $(47,6,6,2;16)$-lottery design. So the variables \lstinline{N} and \lstinline{UB} are bound to $47$ and $17$, respectively, Now \lstinline{bound_isolated_blocks_and_num_shans( N,UB,Rs,RSPairs )} binds \lstinline{RSPairs} to the list \lstinline{[[0,0], [0,1], [0,2], [0,3], [0,4], [1,0], [1,1], [2,0]]} representing possible pairs $(r,s)$ where $r$ is the number of isolated blocks and $s$ is the number of $k$-Shannon subhypergraphs. Fix one such pair, and assume it is possible to take one vertex from each of the asssociate blocks and complete to an independent set of size $5$. The set of ways of doing this is computed by \lstinline{get_deltaI_exceptions} and the possibilities bound to \lstinline{DeltaExceptions}. Assume $(r,s)=(1,1)$, for example. Take $r=1$ vertex from the  isolated blocks, and $s=1$ vertex from the $k$-Shannon subhypergraphs, say $x_1$ and $x_2$; these have degrees $1$ and $2$ respectively. Now \lstinline{get_deltaI_exceptions} uses \cref{ninthdeg2} to establish that there must exist an independent set $I$ extending $\{x_1,x_2\}$, with $\delta(I)=(1,2,2,2,2)$. We now want to see if this can be consistent with the data in \cref{table:exs}. \cref{excessfromisolated} implies $\E(X)=8$, and \cref{mintoes} tells us that there must be at least $28$ $I$-toes. So \lst{populate_toes_in_Iblocks( DeltaTail,MinToes,Excess,Vs ) } is run, with \lstinline{DeltaTail=[2,2,2]}, \lstinline{MinToes=28}, \lstinline{Excess=8}, but this is found to be infeasible by the solver. (The other cases in \lstinline{RSPairs} are even easier to dismiss.) Finally, \lstinline{get_bad_RS_tuples} confirms that \cref{1s2s3s} does always hold with $\pi=5$ for the entries in \lst{RSPairs} and so the proof in case $n=47$ is complete.

\section{Configurations for minimal lottery designs}\label{sec:configs}
In all cases for $30\leq n\leq 70$, each minimal $(n,6,6,2)$ lottery design is achieved through a disjoint union of covering designs appearing in \cref{theconfigs}. (There may be several ways to do this and the table gives just one.) For example, when $n=54$, the configuration $(A,A,E,E,E)$ gives rise to $23$ tickets by the disjoint union of two diagrams of type $\mathbf{(A)}$ and three of type $\mathbf{(E)}$ from \cref{tickpics}. 

Let us be more explicit. Recall the \emph{Fano plane}, or projective plane of order $2$---perhaps the most well-known finite geometry;  this is depicted as $\mathbf{(E)}$ in \cref{theconfigs}. It contains $7$ `lines' (one being represented by a circle) that satisfy the property that any two points lie in exactly one line, and two lines intersect in exactly one point. Since there are $3$ points on each line, we see in particular that the Fano plane is a $(7,3,2)$-covering design; and we have made it into a $(14,6,2)$-covering design by having each point represent a pair of vertices. A set of blocks (tickets) may be read off from diagram $\mathbf{(E)}$ by concatenating the labels on the points in each of the $7$ lines. So each diagram $\mathbf{(E)}$ contributes $7$ blocks to the minimal design and the $\mathbf{(E)}$ diagram accounts for $14$ distinct vertices (or balls). On the other hand, each $\mathbf{(A)}$ diagram contributs just one ticket, containing $6$ vertices; note that $54=6+6+14+14+14$ and $23=1+1+7+7+7$. Now, to recover a complete set of tickets for any given value of $n$, one may simply write the numbers $1,\dots,n$ in any order above a blank set of the $5$ appropriate diagrams. We do this for $n=59$ in \cref{27tix} below. 

\begin{figure}
\parbox{0.4\textwidth}{\begin{tikzpicture}[scale=0.85,
  mydot/.style={
     draw,
     circle,
     fill=black,
     inner sep=3pt}
]
 \draw
 (0,0) coordinate (A) --
 (5,0) coordinate (B);

\node[mydot,label=above: {1,2,3}] at (A) {};
\node[mydot,label=above: {4,5,6}] at (B) {};

\node[label=below: {(1,2,3,4,5,6)}] at (2.5,0) {};
\node[label=above: {\bf (A)}] at (0,2) {};
\end{tikzpicture}}\hfill
\parbox{0.47\textwidth}{
\begin{tikzpicture}[scale=0.85,
  mydot/.style={
     draw,
     circle,
     fill=black,
     inner sep=3pt}
]
 \draw
 (0,0) coordinate (A) --
 (5,0) coordinate (B) --
 ($ (A)!.5!(B) ! {sin(60)*2} ! 90:(B) $) coordinate (C) -- cycle;
\coordinate (O) at
   (barycentric cs:A=1,B=1,C=1);

\node[mydot,label=left: {5,6}] at (A) {};
\node[mydot,label=right: {7,8}] at (B) {};
\node[mydot,label=left: {3,4}] at (C) {};
\node[mydot,label=below: {1,2}, label=above: {$\{\varnothing\}$}] at (O) {};
\node at (0,4.5) {\bf (B)};
\node at (5,3.5) { \begin{tabular}{ l  }
    (1,2,3,4,5,6)\\
    (1,2,3,4,7,8)\\
    (1,2,5,6,7,8)
   \end{tabular} };
\end{tikzpicture}}\\
\ \vspace{10mm}\ 

\parbox{0.45\textwidth}{
\begin{tikzpicture}[scale=0.85,
  mydot/.style={
     draw,
     circle,
     fill=black,
     inner sep=3pt}
]
 \draw
 (0,0) coordinate (A) --
 (5,0) coordinate (B) --
 ($ (A)!.5!(B) ! {sin(60)*2} ! 90:(B) $) coordinate (C) -- cycle;

\node[mydot,label=left: {4,5,6}] at (A) {};
\node[mydot,label=right: {7,8,9}] at (B) {};
\node[mydot,label=left: {1,2,3}] at (C) {};
\node at (5,3.5) {\begin{tabular}{ l  }
    (1,2,3,4,5,6)\\
    (1,2,3,7,8,9)\\
    (4,5,6,7,8,9)
   \end{tabular}};
\node at (0,4.5) {\bf (C)};
\end{tikzpicture}}\hfill\parbox{0.51\textwidth}{

\begin{tikzpicture}[scale=0.85,
  mydot/.style={
     draw,
     circle,
     fill=black,
     inner sep=3pt}
]
    \node[shape=circle,draw=black,fill=black,label=left: {1,2}] (A) at (0,0) {};
    \node[shape=circle,draw=black,fill=black,label=above: {3,4}] (B) at (3*4/4,3*2/4) {};
    \node[shape=circle,draw=black,fill=black,label=above left: {5,6}] (C) at (3*4/4,0) {};
    \node[shape=circle,draw=black,fill=black,label=below: {7,8}] (D) at (3*4/4, 3*-2/4) {};
    \node[shape=circle,draw=black,fill=black,label=right: {9,10}] (E) at (3*8/4,0) {};
    \node at (0,2) {\bf (D)};
    \node at (5.8,-2.3) {\begin{tabular}{ l  }
        (1,2,3,4,9,10)\\
        (1,2,5,6,9,10)\\
        (1,2,7,8,9,10)\\
        (3,4,5,6,7,8)
       \end{tabular}};

    \path [-] (B) edge node {} (C);
    \path [-] (C) edge node {} (D);
    \path [-] (A) edge node {} (C);
    \path [-] (C) edge node {} (E);

    \path [-] (A) edge[bend left=20] node {} (B); 
    \path [-] (B) edge[bend left=20] node {} (E); 

        \path [-] (A) edge[bend right=20] node {} (D); 
    \path [-] (D) edge[bend right=20] node {} (E); 
\end{tikzpicture}}
\ \\\vspace{5mm}\ 

 \begin{center}
\begin{tikzpicture}[scale=0.85,
  mydot/.style={
     draw,
     circle,
     fill=black,
     inner sep=3pt}
]
 \draw
 (0,0) coordinate (A) --
 (5,0) coordinate (B) --
 ($ (A)!.5!(B) ! {sin(60)*2} ! 90:(B) $) coordinate (C) -- cycle;
\coordinate (O) at
   (barycentric cs:A=1,B=1,C=1);
\draw (O) circle [radius= 5*1.717/6];
\draw (C) -- ($ (A)!.5!(B) $) coordinate (LC); 
\draw (A) -- ($ (B)!.5!(C) $) coordinate (LA); 
\draw (B) -- ($ (C)!.5!(A) $) coordinate (LB);

\node[mydot,label=left: {9,10}] at (A) {};
\node[mydot,label=right: {13,14}] at (B) {};
\node[mydot,label=left: {1,2}] at (C) {};
\node[mydot,label=left: {7,8}] at (O) {};
\node[mydot,label=below: {11,12}] at (LC) {};
\node[mydot,label=right: {5,6}] at (LA) {};
\node[mydot,label=left: {3,4}] at (LB) {};
\node at (0,4.5) {\bf (E)};
\end{tikzpicture}
\raisebox{24mm}{
\begin{tabular}{ l  }
 (1,2,3,4,9,10)\\
 (1,2,5,6,13,14)\\
 (1,2,7,8,11,12)\\
 (3,4,5,6,11,12)\\
 (3,4,7,8,13,14)\\
 (5,6,7,8,9,10)\\
 (9,10,11,12,13,14)
\end{tabular}}\end{center}
\vspace{10pt}
\begin{center}
\begin{tikzpicture}[scale=0.85,
    mydot/.style={
       draw,
       circle,
       fill=black,
       inner sep=3pt}
  ]
      \node[shape=circle,draw=black,fill=black,label=left: {1,2,3}] (A) at (0,0) {};
      \node[shape=circle,draw=black,fill=black,label=above: {4,5,6}] (B) at (3*4/4,3*2/4) {};
      \node[shape=circle,draw=black,fill=black,label=below: {7,8,9}] (D) at (3*4/4, 3*-2/4) {};
      \node[shape=circle,draw=black,fill=black,label=right: {10,11,12}] (E) at (3*8/4,0) {};
      \node at (0,2) {\bf (F)};

      \draw (B) -- (D);
      \draw (A) -- (E);
      \draw (A) -- (B);
      \draw (A) -- (D);
      \draw (B) -- (E);
      \draw (D) -- (E);
  
  
  \end{tikzpicture}
  \raisebox{24mm}{\begin{tabular}{ l  }
    (1,2,3,4,5,6)\\
    (1,2,3,7,8,9)\\
    (1,2,3,10,11,12)\\
    (4,5,6,7,8,9)\\
    (4,5,6,10,11,12)\\
    (7,8,9,10,11,12)
   \end{tabular}}

\vspace{10pt}
\textbf{(G)} as in \textbf{(E)} but replacing each occurrence of $14$ with a number from $1$ to $13$ which leaves a valid ticket.
\end{center}
\caption{Ticket configurations}\label{tickpics}\end{figure}

\begin{thm}\cref{theconfigs} lists $j=L(n,6,6,2)$ together with configurations described in \cref{tickpics} which afford an $(n,6,6,2;j)$-lottery design.\label{alltogetherthm}\end{thm}

\begin{table}[ht]
    \begin{center}
\begin{tabular}{ |c|c|c| } 
  \hline
 $n$ & $L(n,6,6,2)$ & Configuration \\ 
 \hline
 32 & 7 & $(A,A,A,A,B)$ \\ 
 33 & 7 & $(A,A,A,A,C)$  \\ 
34 & 8 & $(A,A,A,A,D)$  \\
35 & 9 & $(A,A,A,B,C)$  \\
36 & 9 & $(A,A,A,C,C)$  \\
37 & 10 & $(A,A,A,C,D)$  \\
38 & 11 & $(A,A,A,A,E)$  \\
39 & 11 & $(A,A,C,C,C)$  \\
40 & 12 & $(A,A,C,C,D)$  \\
41 & 13 & $(A,A,A,C,E)$  \\
42 & 13 & $(A,C,C,C,C)$  \\
43 & 14 & $(A,C,C,C,D)$  \\
44 & 15 & $(A,A,C,C,E)$  \\
45 & 15 & $(C,C,C,C,C)$  \\
46 & 16 & $(C,C,C,C,D)$  \\
47 & 17 & $(A,C,C,C,E)$  \\
48 & 18 & $(A,C,C,D,E)$  \\
49 & 19 & $(A,A,C,E,E)$  \\
50 & 19 & $(A,A,A,A,E)$  \\
51 & 20 & $(A,A,A,D,E)$  \\
  \hline
\end{tabular}
\begin{tabular}{ |c|c|c| } 
  \hline
 $n$ & $L(n,6,6,2)$ & Configuration \\ 
 \hline
52 & 21 & $(A,C,C,E,E)$  \\
53 & 22 & $(A,C,D,E,E)$  \\
54 & 23 & $(A,A,E,E,E)$  \\
55 & 23 & $(C,C,C,E,E)$  \\
56 & 24 & $(C,C,D,E,E)$  \\
57 & 25 & $(A,C,E,E,E)$  \\
58 & 26 & $(A,D,E,E,E)$  \\
59 & 27 & $(B,C,E,E,E)$  \\
60 & 27 & $(C,C,E,E,E)$ \\
61 & 28 & $(C,D,E,E,E)$ \\
62 & 29 & $(D,D,E,E,E)$ \\
63 & 30 & $(C,E,E,E,F)$ \\
64 & 31 & $(D,E,E,E,F)$ \\
65 & 31 & $(C,E,E,E,E)$ \\
66 & 32 & $(D,E,E,E,E)$ \\
67 & 34 & $(E,E,E,F,G)$ \\
68 & 34 & $(E,E,E,E,F)$ \\
69 & 35 & $(E,E,E,E,G)$ \\
70 & 35 & $(E,E,E,E,E)$ \\
&&\\
  \hline
\end{tabular}
\caption{Lottery numbers and their configurations}
\label{theconfigs}
\end{center}
\end{table}

\begin{center}
    \begin{table}
\scalebox{0.9}{\small
\begin{tabular}{| c c c c c |}\hline
\cellcolor{lightgray} $1,2,3,4,5,6$ & $9,10,11,12,13,14$ & \cellcolor{lightgray} $18,19,20,21,26,27$ & $32,33,34,35,40,41$ & \cellcolor{lightgray} $46,47,48,49,54,55$\\ 
\cellcolor{lightgray} $1,2,3,4,7,8$ & $9,10,11,15,16,17$ & \cellcolor{lightgray} $18,19,22,23,30,31$ & $32,33,36,37,44,45$ & \cellcolor{lightgray} $46,47,50,51,58,59$\\
\cellcolor{lightgray} $1,2,5,6,7,8$ & $12,13,14,15,16,17$ & \cellcolor{lightgray} $18,19,24,25,28,29$ & $32,33,38,39,42,43$ & \cellcolor{lightgray}  $46,47,52,53,56,57$ \\
\cellcolor{lightgray}  &  & \cellcolor{lightgray} $20,21,22,23,28,29$ & $34,35,36,37,42,43$ & \cellcolor{lightgray} $48,49,50,51,56,57$\\
\cellcolor{lightgray}  &  & \cellcolor{lightgray} $20,21,24,25,30,31$ & $34,35,38,39,44,45$ & \cellcolor{lightgray} $48,49,52,53,58,59$\\
\cellcolor{lightgray}  &  & \cellcolor{lightgray} $22,23,24,25,26,27$ & $36,37,38,39,40,41$  & \cellcolor{lightgray} $50,51,52,53,54,55$\\
\cellcolor{lightgray}  &  & \cellcolor{lightgray} $26,27,28,29,30,31$ & $40,41,42,43,44,45$ & \cellcolor{lightgray} $54,55,56,57,58,59$\\\hline
\end{tabular}
}\caption{One set of $27$ tickets for $n=59$ balls using configuration $(B,C,E,E,E)$.\label{27tix}}\end{table}
\end{center}

\section{Conclusion} Using constraint programming, we calculated a large set of new lottery design numbers: minimal configurations of tickets that guarantee winning a prize under common lottery rules. In doing so we hope to have offered a basic blueprint for an approach to modelling combinatorial designs in CP, where it is clear that naive approaches would be unviable. Namely, we show the fruitfulness of eschewing a global model of a combinatorial design in favour of a local one. By this we mean that one should not seek to design the variables of one's model to satisfy the direct definition of the required design; but instead to focus attention on a well-chosen small piece of it---one that is characterised by a property relevant to detecting a global {\small\sc unsat} conclusion. Developing and analysing the constraints that the definitions place on the local data should then in reasonable time allow the solver to achieve {\small\sc unsat} or at least return a drastically reduced set of possibilities from which further analysis can proceed. In effect, what we propose is a general approach to symmetry breaking in combinatorial designs that is far more powerful than one would get from a typical lex-chain constraint. 

We note that there are dozens of different general types of designs in \cite{Hand} leading to tens of thousands of different constraint problems to model. We hope that deploying a local-based strategy of the sort described here could generate significant improvements to the state of knowledge thereof.

{\bf{Declarations:}} The authors are supported by the Leverhulme Trust Research Project Grant number RPG-2021-080. The authors have no competing interests to declare that are relevant to the content of this article.

{\bf{Acknowledgement:}} We would like to thank Leo Storme for helpful comments and corrections on an earlier version. Many thanks also to the referee and editor who helped hone the paper to be useful for the readership of this journal.
 
\AtNextBibliography{\footnotesize} 
\printbibliography

\end{document}